\documentclass{amsart}
\usepackage{amsmath,amsthm,amsfonts,amssymb,amscd}
\usepackage{psfrag,graphicx}


\theoremstyle{plain}
\newtheorem{main}{Theorem}

\newtheorem{Thm}{Theorem}[section]
\newtheorem{Lem}[Thm]{Lemma}
\newtheorem{Prop}[Thm]{Proposition}
\newtheorem{Cor}[Thm]{Corollary}
\theoremstyle{remark}
\newtheorem{Def}[Thm] {Definition}
\newtheorem{Rem}[Thm] {Remark}
\newtheorem{Con}{Conjecture}

\long\def\begcom#1\endcom{}

\newcommand{\quand}{\quad\text{and}\quad}

\def\cU{{\mathcal U}}

\begin{document}

\title[Continuity of topological entropy]
{Continuity of topological entropy \\ for perturbation of time-one maps of \\ hyperbolic flows}

\author{Radu Saghin and Jiagang Yang}

\thanks{Instituto de Matem\'atica, Pontific\'ia Universidad Cat\'olica de Valpara\'iso, Blanco Viel 596, Valpara\'iso, Chile}
\thanks{Departamento de Geometria, Instituto de Matem\'atica e Estat\'istica, Universidade Federal Fluminense,
Niter\'oi, Brazil}

\email{rsaghin@gmail.com, yangjg@impa.br}

\date{today}

\thanks{RS was partially supported by Anillo CONYCIT ACT-1103 and FONDECYT regular 1130611; JY was partially supported by CNPq,
FAPERJ and Palis-Balzan project}

\begin{abstract}
We consider a $C^1$ neighborhood of the time-one map of a hyperbolic flow and prove that the topological entropy varies continuously for diffeomorphisms in this neighborhood.
This shows that the topological entropy varies continuously for all known examples of partially hyperbolic diffeomorphisms with one-dimensional center bundle.
\end{abstract}

\maketitle


\section{Introduction}\label{s.introduction}

In this paper we are addressing the problem of continuity of the topological entropy on the set of partially hyperbolic diffeomorphisms with the dimension of the center equal to one, and in particular for perturbations of time-one maps of hyperbolic flows. The main theorem which we obtain is the following.

\begin{main}\label{t.mainA}

Let $\phi_t$ be a $C^1$ Anosov flow on a compact Riemannian manifold $M$. There exists a $C^1$ neighborhood $\cU$ of $\phi_1$, in the set of $C^1$ diffeomorphisms of $M$, such that the topological entropy varies continuously in $\cU$.  
\end{main}

The topological entropy, together with the metric entropy, is an important invariant of a dynamical system (see the definition in the following section), which describes the complexity of the orbits. The dependence of the topological entropy on the map was extensively studied by various mathematicians (Newhouse, Yomdin, Katok, etc.). There are several important factors which contribute to this dependence, namely the manifold, and the space of maps which is considered.

Let us remark that in the space of continuous maps, generically the topological entropy is infinity (\cite{Yano}), so it make sense to study the problem of continuity at least in the $C^1$ topology.

The upper semicontinuity of the topological entropy seems to depend on the regularity of the maps and on the existence of homoclinic tangencies. By Yomdin \cite{Yomdin} (see also \cite{Buzzi}), the topological entropy is upper semicontinuous in the space of $C^{\infty}$ diffeomorphisms. And by Misiurewicz \cite{Mis73}, the topological entropy is not upper semicontinuous in the $C^r$ topology, for any $r<\infty$. The examples where the upper semicontinuity fails are diffeomorphisms with homoclinic tangencies; an arbitrary $C^r$ small perturbation can create a horseshoe at the homoclinic orbit, such that the topological entropy increases (proportionally with $\frac 1r$). We note that, by the result of Liao, Viana and Yang \cite{LVY}, the topological entropy is in fact upper semicontinuous on the space of $C^1$ diffeomorphisms away from tangencies.

The lower continuity of the topological entropy seems to depend on the existence of horseshoes or other types of invariant sets which are robust and ``carry'' the entropy (for example uniformly hyperbolic sets), and in many cases this depends on the dimension of the manifold $M$. In dimension one, the topological entropy is lower semicontinuous on the space of $C^0$ maps and locally constant (zero) on the space of  homeomorphisms, while in dimension two it is lower semicontinuous on the space of $C^{1+\alpha}$ diffeomorphisms, by a result of Katok (\cite{Katok}). The reason for these results is that in lower dimensions the topological entropy is due to the existence of horseshoes, and these horseshoes are persistent after small perturbations. In dimension greater or equal to three one can easily construct examples of $C^{\infty}$ diffeomorphisms where the lower semicontinuity of the topological entropy fails, using partially hyperbolic invariant sets with positive entropy that can disappear after an arbitrarily small perturbation.

Of course, the classical example of diffeomorphisms which are both away from tangencies and have persistent horseshoes are the uniformly hyperbolic maps, in this case the topological entropy is not only continuous, it is locally constant in the $C^1$ topology due to the structural stability. The partially hyperbolic diffeomorphisms, as a generalization of uniformly hyperbolic maps,
were proposed by Brin, Pesin~\cite{BP74} and by Pugh, Shub~\cite{PS72} independently in the early 1970's. A diffeomorphism 
$f$ is {\it partially hyperbolic}, if there exists a decomposition $TM = E^s \oplus E^c \oplus E^u$
of the tangent bundle $TM$ into three continuous sub-bundles $x\mapsto E^s_x$, $x\mapsto E^c_x$ and $x\mapsto E^u_x$ such that
\begin{itemize}
\item all three sub-bundles $E^s, E^c$ and $E^u$ are invariant under the derivative $Df$;
\item $Df \mid E^s$ is a uniform contraction, $Df\mid E^u$ is a uniform expansion and $Df \mid E^c$ lies in between them: there exists $\mu<1$ such that
$$
\frac{\|Df(x)v^s\|}{\|Df(x)v^c\|} \le \mu
\quand
\frac{\|Df(x)v^c\|}{\|Df(x)v^u\|} \le \mu
$$
for any unit vectors $v^s\in E^s$, $v^c\in E^c$ and $v^u\in E^u$ and any $x\in M$.
\end{itemize}
Partially hyperbolic diffeomorphism form an open subset of the space of $C^r$ diffeomorphisms of $M$, for any $r\ge 1$.

A natural question to ask is how is the dependence of the topological entropy on the space of partially hyperbolic diffeomorphisms with the dimension of the center equal to one or two, since the topological entropy is constant zero for homeomorphisms in dimension one and continuous for $C^{\infty}$ diffeomorphisms in dimension two. An example of a $C^{\infty}$ partially hyperbolic diffeomorphisms with 2-D center where the topological entropy is not lower semicontinuous can be found in \cite{HSX} (an invariant set with high entropy can disappear after an arbitrarily small perturbation). Thus, we will focus on the partially hyperbolic diffeomorphisms with one dimensional center, and a natural conjecture formulated in this setting is the following.

\begin{Con}\label{c.1d}

The topological entropy is continuous on the space of $C^1$ partially hyperbolic diffeomorphisms with the dimension of the center equal to one.

\end{Con}

An easier version would be to restrict the attention to the space of three dimensional manifolds, where the stable, center,
and unstable bundle of a partially hyperbolic diffeomorphism are automatically one dimensional.

\begin{Con}\label{c.continuiation}

The topological entropy depends continuously in the space of 3 dimensional partially hyperbolic diffeomorphisms.

\end{Con} 

We remark that these conjectures are known to be true for all the known examples of partially hyperbolic diffeomorphisms with one
dimensional center or in dimension three. A list with the known partially hyperbolic diffeomorphisms with one dimensional center is the following:
\begin{enumerate}
\item
uniformly hyperbolic diffeomorphisms
\item
skew products over uniformly hyperbolic, with the fiber being a circle, and perturbations (\cite{BW, P, H13b,HP14})
\item
derived from Anosov diffeomorphisms (\cite{M,HP14,Potrie})
\item
non-dynamically coherent examples of Hertz-Hertz-Ures \cite{HHU},
\item
time-one maps of Anosov flows and perturbations
\item
new examples by Bonatti-Parwani-Potrie \cite{BPP}
\end{enumerate}

In the first 4 examples, the topological entropy is locally constant. In the cases (5) and (6),
the topological entropy is not locally constant. In fact, in case (5), let $X_t$ be a $C^{\infty}$ Anosov flow, and considering the smooth family of time-$t$ diffeomorphisms 
$f_t=X_t$, we observe that $h(f_t)=th(X)$ is not constant. The diffeomorphism in case (6) has a unique attracting set and a unique expanding set, and restricted to the attracting set (resp. expanding set) the dynamic is that of (a perturbation of) the time-one map of a hyperbolic flow restricted to an attractor (resp. repeller), hence the discussion is similar to the case (5) from above for time-one maps of Anosov flows. More details on case (6) will be given in Appendix.
The continuity of the topological entropy at the time-one map of a transitive Anosov flow was proven by Hua in her thesis \cite{Hua}, and for some specific perturbations by Hu-Zhu \cite{HZ}; however, they only show the continuity at one point (the time-one map of the flow), and not for a neighborhood. The result of our paper deals with the case of perturbations of the time-one maps of Anosov flows, and similar arguments can be used also in the case of Bonatti-Parwani-Potrie examples, which will be explained in Appendix, thus completing all the known cases mentioned above.

Our proof of the  lower semicontinuity of topological entropy is different from the method used in most other papers, namely the construction of horseshoes, because in our case they may not exist. The main idea of our paper (and the most difficult part) is to show that the unstable foliation has uniform expansion on every unstable disk. Under these conditions of uniformity one can show that this expansion cannot drop much after small perturbations, and the connection with the topological entropy will then help us get the final conclusion. In this last argument the fact that the center bundle has dimension one is essential.

The paper is organized as follows. In Section 2 we introduce the definitions and various preliminary results, while in Section 3 we give the proof of Theorem \ref{t.mainA}. In the Appendix we include a discussion on the other known examples of partially hyperbolic diffeomorphisms with the dimension of the center equal to one.

\section{Definitions and Preliminary results}

\subsection{Time-one maps of Anosov flows and perturbations}

We start by introducing the notion of hyperbolic flow and the spectral decomposition.

\begin{Def}
A $C^1$ flow $\phi$ on the compact Riemannian manifold $M$ is {\it Anosov} (or {\it uniformly hyperbolic}) if there exists a splitting of the tangent bundle $TM=E^s\oplus E^c\oplus E^u$ which is invariant under the flow, $E^c$ is the direction of the flow, and there exist $C,\lambda>0$ such that
$$
\|D\phi_t(x)v^s\|\leq C\exp^{-\lambda t}\| v^s\|,\ \ \forall x\in M,\ t>0,\ v^s\in E^s_x;
$$
$$
\|D\phi_t(x)v^u\|\geq C^{-1}\exp^{\lambda t}\| v^u\|,\ \ \forall x\in M,\ t>0,\ v^u\in E^u_x.
$$

The map $\phi_t:M\rightarrow M$ is the {\it time-$t$ map of the flow $\phi$}.
\end{Def}

We remark that the time-$t$ map of an Anosov flow is a partially hyperbolic diffeomorphism, meaning that the tangent bundle has an invariant splitting $TM=E^s\oplus E^c\oplus E^u$ (in this case the same one of the Anosov flow), and $D\phi_t$ is uniformly contracting on $E^s$, uniformly expanding on $E^u$, and the contraction (expansion) on $E^c$ is strictly smaller than the one on $E^s$ ($E^u$).

The invariant bundles $E^s,E^u,E^c,E^s\oplus E^c$ and $E^u\oplus E^c$ integrate to unique foliations $W^s,W^u,W^c,W^{cs}$ and $W^{cu}$.

\begin{Thm}{\bf (Spectral Decomposition for Flows)}
If $\phi$ is an Anosov flow on the compact manifold $M$, then there exists a decomposition of the non-wandering set of $\phi$, $NW(\phi)=\Lambda_1\cup\Lambda_2\cup\dots\cup\Lambda_k$, where the sets $\Lambda_i$ are mutually disjoint, and each $\Lambda_i$ is a transitive compact invariant set for $\phi$, such that $\Lambda_i=\overline{W^{cs}(x)}\cap\overline{W^{cu}(x)}$, $\forall x\in\Lambda_i$. 
\end{Thm}

Throughout the paper, we fix the Anosov flow $\phi$ which satisfies the above theorem.
We turn our attention now to perturbations of the time-one map of an Anosov flow. A partially hyperbolic diffeomorphism is dynamically coherent if the center-stable, central and center-unstable bundles integrate to invariant foliations $W^{cs}$, $W^{c}$ and $W^{cu}$
respectively, the center-stable foliation is sub-foliated by the center and the stable foliations, and the center-unstable foliation is sub-foliated by center and unstable foliations.

Since the center foliation is smooth for $\phi_1$, which is plaque expansive, hence, by Theorem 7.1 of Hirsch, Pugh and Shub (see \cite{HPS}), there exists a $C^1$ neighborhood of $\phi_1$, such that every diffeomorphism in this neighborhood is partially hyperbolic with one dimensional center, and is dynamically coherent.

\begin{Prop}\label{centerconjugation}[\cite{HPS}[Theorem 7.1]]

There exists a neighborhood $\cU$ of $\phi_1$ in the space of $C^1$ diffeomorphisms, such that for each $f\in\cU$, there 
is a homeomorphism $h_f:M\rightarrow M$, which is a leaf conjugacy, meaning that $h_f(W^c(x,\phi_1))=W_f^c(h_f(x),f)$, and with $h_f$
close to $Id$ in the $C^0$ topology, where $W^{*}_{f}$ denotes the $*$ foliation of $f$ ($*=s,u,c,cs,cu$).

\end{Prop}

It is easy to see that the product structure of each basic set is preserved by the leaf conjugacy: the center stable leaves and center
unstable leaves are mapped to the corresponding cente stable leaves and center unstable leaves.

\begin{Cor}\label{productstructure}

For any $f\in\cU$ and any $1\leq i\leq k$, $h_f(\Lambda_i)$ has local product structure, meaning that there exists $\beta_0>0$ such that if $x,y\in h_f(\Lambda_i)$, $d(x,y)<\beta<\beta_0$, then $W^u_{2\beta}(x,f)\cap W^{cs}_{2\beta}(y,f)$ and $W^s_{2\beta}(x,f)\cap W^{cu}_{2\beta}(y,f)$ consist both of exactly one point which belongs to $h_f(\Lambda_i)$.

\end{Cor}

By the leaf conjugation, $f\in \cU$ preserves each center leaf. For the flow $\phi$, denote by $I(x)$ the flow segment
$(\phi_t(x))_{t\in[0,1]}$. Then by the continuity, for each $f\in \cU$, $f(x)$ is close to $\phi_1(x)$, and the segment
$I_f(x)\subset W^c(x,f)$ between $x$ and $f(x)$ is close to $I(x)$ in the Hausdorff topology. Moreover, $W^c(x,f)=\cup_{i} f^i(I_f(x))$ and the length of $I_f(x)$ is bounded uniformly from
above and below by some positive numbers $K_1$ and $K_2$, which are independent of $x$.

Let $d_*$ be the distance between two points measured along $W^*$, and let $W^*_{r}(x,f)$ denote the ball of radius $r$ centered at $x$ in the *-manifold of $x$ with respect to $f$, where $*\in\{ s,c,u\}$. 

As a consequence of the above discussion, one can easily show that

\begin{Lem}\label{nonexpandingcenter}

$f$ is not expanding on the center leaves, meaning that there exist $K_1,K_2>0$ such that if $y\in W^c(x,f)$, with $d_c(x,y)\leq lK_1$ for some $l>0$, then
$$
d_c(f^n(x),f^n(y))\leq lK_2,\ \forall n\geq 0.
$$
The same is true for $f^{-1}$.

\end{Lem}

The following are various properties of the perturbations of the time-one map of an Anosov flow.

\begin{Prop}\label{nonwandering}

The non-wandering set of $f\in\cU$ is contained in $\cup_{i=1}^kh_f(\Lambda_i)$. 
Each set $h_f(\Lambda_i)$ is saturated by center leaves and invariant under $f$.

\end{Prop}
\begin{proof}
We only need to check that the non-wandering set of $f$ is contained in $\cup_{i=1}^kh_f(\Lambda_i)$, the rest is an easy
corollary of Proposition \ref{centerconjugation}.

Fix $x\notin \cup_{i=1}^k \Lambda_i$, it suffices to show that $h_f(x)$ does not belong to the non-wandering set of $f$.
Because $x$ is wandering, there is a neighborhood $U$ such that for each $y\in U$, its flow orbit intersects $U$ with a unique
connected component. This property is preserved clearly by the leaf conjugacy, that is, for every $h_f(y)\in h_f(U)$,
$W^c_f(h_f(y))$ intersects $h_f(U)$ with a unique connected component with a small length. Then by Lemma \ref{nonexpandingcenter}, 
$f(h_f(y))\notin h_f(U)$. This implies that $h_f(x)$ is a wandering point of $f$.

\end{proof}

\begin{Prop}\label{centerholonomy}
The center holonomy between stable and unstable leaves along center-stable and respectively  center-unstable leaves is globally defined, for every $f\in\cU$: if $y\in W^c(x,f)$, then there exist the continuous holonomies $h^c_s:W^s(x,f)\rightarrow W^s(y,f)$ and $h^c_u:W^u(x,f)\rightarrow W^u(y,f)$ (we will use the notation $h^c$ if no confusion can be made).
\end{Prop}

\begin{proof}

First observe that restricted to any center unstable leaf, locally the center holonomy map is uniformly continuous:

\begin{Lem}\label{localholonomy}
For any $K>0$, there exist $\epsilon_K$ such that if $f\in\cU$, $y\in W^c_K(x,f)$, then the center holonomy map $h^c$ is well defined from $W^*_{\epsilon_K}(x,f)$ to $W^*(y,f)$, $*\in\{ s,u\}$. Moreover, for any $0<\epsilon<\epsilon_K$, there are $c_1(\epsilon,K),c_2(\epsilon,K)>0$ such that:
$$
W^*_{c_1(\epsilon,K)}(y,f)\subset h^c(W^*_{\epsilon}(x,f)\subset W^*_{c_2(\epsilon,K)}(y,f).
$$
Furthermore $\lim_{\epsilon\rightarrow 0}c_1(\epsilon,K)=\lim_{\epsilon\rightarrow 0}c_2(\epsilon, K)=0$.
\end{Lem}

\begin{proof}

The holonomy is clearly continuous, and by compactness it is uniformly continuous.

\end{proof}

Now let us continue the proof.

Let $f\in\cU$, and $x,y \in M$, with $y\in W^c(x,f)$. For any $N>0$, we will construct the center holonomy map 
from $W^u_N(x,f)$ to $W^u(y,f)$.

We consider $f^{-n}$ for $n$ very large, then on one hand, the size of $f^{-n}(W^u_N(x,f))$ is very small, and on the other hand $d^c(f^{-n}(x),f^{-n}(y))$ is bounded uniformly from above by a number $K$, by Lemma \ref{nonexpandingcenter}. By (uniform) continuity of the center foliation of Lemma \ref{localholonomy}, we can construct the center-holonomy map $h^c_{n}$ from $W^u_{\epsilon_K}(f^{-n}(x),f)$ to $W^u(f^{-n}(y),f)$, where $\epsilon_K$ is small enough (depending on $K$ but not on $n$). Taking $n$ large enough, we have that $f^{-n}(W^u_N(x,f))\subset W^u_{\epsilon_K}(f^{-n}(x),f)$. Then push forward by $f^n$, and we obtain the holonomy map $h_u^c=f^n\circ h^c_n\circ f^{-n}$ from $W^u_N(x,f)$ to $W^u(y,f)$. Note that by the invariance of the center foliation, $h^c_{n-1}= f\circ h^c_n\circ f^{-1}$. We complete the proof by letting $N\rightarrow \infty$. The case of the center holonomy between stable leaves is similar.

We remark that the center holonomy map $h^c: W^u(x,f)\to W^u(y,f)$ satisfies that the distance $d_c(z,h^c(z))$ is uniformly bounded from above by a constant depending only on $d_c(x,y)$. This is because for $n$ large enough,
$$d_c(f^{-n}(z), h^c_n(f^{-n}(z)))\approx d_c(f^{-n}(x),f^{-n}(y))$$ 
is bounded uniformly from above, then by Lemma \ref{nonexpandingcenter}, $d_c(z,h^c(z))$ is bounded uniformly from above.

\end{proof}

The following proposition shows that in every center unstable leaf, the center foliation and the unstable foliation hold some kind of
global product structure.

\begin{Prop}\label{complete}

$W^{cu}(x,f)=\cup_{y\in W^u(x,f)}W^c(y,f)=\cup_{y\in W^c(x,f)}W^u(y,f)$ for any $x\in M$ and $f\in \cU$.

\end{Prop}

\begin{proof}

We need a property on the unstable holonomy in each center unstable leaf, which is similar to Proposition \ref{centerholonomy},
hence we only give a sketch of its proof.

\begin{Lem}\label{unstableholonomy}

If $y\in W^u(x,f)$, then there exist the continuous holonomies $h^u_c:W^c(x,f)\rightarrow W^c(y,f)$.

\end{Lem}
\begin{proof}
For $n\to \infty$, $f^{-n}(y)\to f^{-n}(x)$. Hence, there is $K_n\to \infty$, such that $h^u_n$ is the unstable holonomy map from
$W^c_{K_n}(f^{-n}(x),f)$ to $W^c(f^{-n}(y),f)$. Define $h^u_c= f^{n}\circ h^u_n \circ f^{-n}$ the unstable holonomy map from $f^n(W^c_{K_n}(f^{-n}(x),f))$ to $W^c(y,f)$. By Lemma \ref{nonexpandingcenter}, there exist $K_n^{'}\to \infty$ such that
$W^c_{K_n^{'}}(x,f)\subset f^n(W^c_{K_n}(f^{-n}(x),f))$. Hence, $h^u_c$ is well defined.
\end{proof}

Now let us continue the proof. We first prove that 
$$W^{cu}(x,f)=\cup_{y\in W^u(x,f)}W^c(y,f).$$ 
It is clear that
$\cup_{y\in W^u(x,f)}W^c(y,f)\subset W^{cu}(x,f)$.
Suppose by contradiction that $\cup_{y\in W^u(x,f)}W^c(y,f)\subsetneq W^{cu}(x,f)$, take $z\in \partial \cup_{y\in W^u(x,f)}W^c(y,f)$, by the local
product structure of $W^c$ and $W^u$ in the center unstable leaf, it is easy to conclude that $W^c(z)\subset \partial \cup_{y\in W^u(x,f)}W^c(y,f)$ and $z\notin \cup_{y\in W^u(x,f)}W^c(y,f)$. Take $z^{'}\in \cup_{y\in W^u(x,f)}W^c(y,f)$ close to $z$, then $W^u(z^{'},f)\cap W^c(z,f)\neq \emptyset$.
Note that $W^c(z^{'},f)\cap W^u(x,f)\neq \emptyset$. Hence, by the global holonomy of Proposition \ref{centerholonomy}, $W^c(z,f)\cap W^u(x,f)\neq \emptyset$, which contradicts the fact that $z\notin \cup_{y\in W^u(x,f)}W^c(y,f)$.

The proof of $W^{cu}(x,f)=\cup_{y\in W^c(x,f)}W^u(y,f)$ is similar as above, replacing Proposition \ref{centerholonomy}
be Lemma \ref{unstableholonomy}.

The proof of Proposition \ref{complete} is complete.
\end{proof}

The following proposition is a generalization of a result in Plante \cite{Plante}: for every Anosov transtive flow, if the strong unstable manifold of a periodic point is not dense, then this Anosov flow is a suspension flow with constant roof function. 

\begin{Prop}\label{centerunstabledense}
For any $f\in\cU$ there exists $K_0>0$ such that for any $x\in h_f(\Lambda_i)$, we have that $W^c_{K_0}(W^u(x,f),f)$ is dense in $h_f(\Lambda_i)$.

\end{Prop}

\begin{proof}

By Proposition \ref{complete}, for any $x\in h_f(\Lambda_i)$,
$$
\cup_{L\in\mathbb N} (\cup_{z\in W^c_L(W^u_L(x,f),f)}W^s_1(z,f))=\cup_{z\in W^{cu}(x,f)}W^s_1(z,f).
$$ 
Because $\Lambda_i$ is a basic component of the Anosov flow $\phi$, $W^{cs}(h_f^{-1}(x),\phi)$ is dense in $\Lambda_i$.
Then by the center conjugation of Proposition \ref{nonwandering}, $W^{cu}(x,f)$ is dense in $h_f(\Lambda_i)$ and 
$$\{\Gamma_{f,L}(x)=\cup_{z\in W^c_L(W^u_L(x,f),f)}W^s_1(z,f)\}_{L\in \mathbb{N}}$$ 
is an open covering of the $h_f(\Lambda_i)$. Because $h_f(\Lambda_i)$ is compact, we can take $L(x)$ such that $h_f(\Lambda_i)\subset \Gamma_{f,L(x)}(x)$.

By the continuity of invariant foliations, for any $f\in\cU$, and for any $x\in h_f(\Lambda_i)$, 
there exists a small neighborhood $B_x$ of $x$ such that for any $z\in B_x$, 
$h_f(\Lambda_i)\subset \Gamma_{f,L(x)}(z)$. Take an open covering $B_{x_1},\dots,B_{x_k}$ of $h_f(\Lambda_i)$ and let $L=\max\{L(x_1),\dots,L(x_k)\}$. By the above construction, for any $x\in h_f(\Lambda_i)$, we always have 
$h_f(\Lambda_i) \subset \Gamma_{f,L}(x)$.

For any $x\in h_f(\Lambda_i)$, we have that $h_f(\Lambda_i) \subset \Gamma_{f,L}(f^{-n}(x))$, and by the invariance of $h_f(\Lambda_i)$ we get $h_f(\Lambda_i)\subset f^n(\Gamma_{f,L}(f^{-n}(x))$. Because the center is not expanding, there exists a constant $K_0>0$ depending only on $L$ and $\cU$ (not on $n$) such that
$$
f^n(\Gamma_{f,L}(f^{-n}(x)))\subset \cup_{z\in W^c_{K_0}(W^u(x,f),f)}f^n(W^s_1(f^{-n}(z),f)).$$
By the uniform contraction of the strong stable foliation, $W^c_{K_0}(W^u(x,f),f)$ is dense in $h_f(\Lambda_i)$.
The proof is complete

\end{proof}

We remark that it is easy to show that $K_0$ from above can be chosen such that the result of the Proposition \ref{centerunstabledense} is true for any $1\leq i\leq k$ and for any $g$ in a small neighborhood of $f$.

\subsection{Topological entropy}

In this subsection we will discuss about the topological entropy.

\begin{Def}
Let $M$ be a compact metric space, and $f:M\rightarrow M$ a homeomorphism. 
For $K\subset M$ a subset not necessary invariant, the {\it topological entropy of $f$ on $K$} 
is defined as follows. Let $d_n(x,y)=\max _{i=0}^{n-1}d(f^i(x),f^i(y))$. Let $a(f,n,\delta,K)$ be the maximal cardinality of $\delta$-separated set of $K$ in the metric $d_n$, and let $b(f,n,\delta,K)$ be the minimal cardinality of a $\delta$-spanning set of $K$ in the same metric $d_n$. Then
$$
a(f,\delta,K)=\limsup_{n\rightarrow\infty}\frac 1n\log a(f,n,\delta,K),
$$
$$
b(f,\delta,K)=\limsup_{n\rightarrow\infty}\frac 1n\log b(f,n,\delta,K),
$$
and the topological entropy of $f$ on $K$ is
$$
h(f,K)=\lim_{\delta\rightarrow 0}a(f,\delta,K)=\lim_{\delta\rightarrow 0}b(f,\delta,K).
$$
The topological entropy of $f$ is $h(f)=h(f,M)$.
\end{Def}

\begin{Rem}\label{piece}

With the same hypothesis as above, suppose $M=\cup_{i=1}^\infty M_i$ is a measurable partition of $M$, then
$$h(f)=\sup_i h(f,M_i).$$

\end{Rem}

In the case of partially hyperbolic diffeomorphisms with one-dimensional center, and with non-expanding center, the entropy can be calculated using only the unstable manifolds.

\begin{Def}\label{localproductset}

Let $f$ be a dynamically coherent partially hyperbolic diffeomorphism on $M$, then for any $x\in M$, we define
$$
A(x,\delta)=\{ y\in M: \exists x_u\in W^u_{\delta}(x,f), x_c\in W^c_{\delta}(x,f), y=W^c_{loc}(x_u,f)\cap W^u_{loc}(x_c,f)\},
$$
and
$$D(x,\delta)=\cup_{y\in A(x,\delta)}W^s_{\delta}(y,f).$$
\end{Def}

\begin{Prop}\label{unstableentropy}
Let $f$ be a dynamically coherent partially hyperbolic diffeomorphism on the compact manifold $M$, with one dimensional center, and $\Lambda$ an invariant set of $f$ with local product structure. Suppose that there exist $K_1,K_2>0$ such that if $y\in W^c_{K_1}(x)$ then $f^n(y)\in W^c_{K_2}(f^n(x))$, $\forall n\in\mathbb N$. 
Then there is $\delta_0>0$ such that for any $\delta<\delta_0$, $h(f,D(x,\delta)\cap\Lambda)=h(f, W^u_{\delta}(x)\cap\Lambda)$.
\end{Prop}

\begin{proof}

We first prove that $h(f,W^u_{\delta}(x,f)\cap\Lambda)=h(f, A(x,\delta)\cap\Lambda)$. 

One inequality is obvious, since $W^u_{\delta}(x,f)\cap\Lambda\subset A(x,\delta)\cap\Lambda$. 

By the uniform transversity between the two bundles $E^c$ and $E^u$, for any $\delta<\delta_0$ sufficiently small, 
in the definition of $A(x,\delta)$ above, $y=W^c_{\delta}(x_u,f)\cap W^u_{\delta}(x_c,f)$. We can choose $\delta_0<K_1$.
Note that Lemma \ref{localholonomy} is still true in this case, let $c(\epsilon)=c_2(\epsilon,K_1)$ be the constant given there.

Let 
$S(n,\epsilon)$ be an $(n,\epsilon)$-spanning set of $W^u_{\delta}(x)\cap\Lambda$ of minimal cardinality. For any $y\in S(n,\epsilon)$, 
let $T(y,n,\epsilon)$ be a $(n,c(\epsilon))$-spanning set of $W^c_{K_1}(y)\cap A(x,\delta)$ of minimal cardinality. We claim that
the cardinality of $T(y,n,\epsilon)$ is bounded by $\frac{2nK_2}{c(\epsilon)}$. 
\begin{proof}
This is because the center manifold is one dimensional, and by hypothesis, $f^n(W^c_{K_1}(y,f))\subset W^c_{K_2}(f^n(y),f)$. 
For each $0\leq i \leq n$, we can take $T(f^i(y),0,\epsilon)$ a $(0,c(\epsilon))$-spanning set of $W^c_{K_2}(f^i(y),f)$ with cardinality bounded by $\frac{2K_2}{c(\epsilon)}$. Then
$\cup _{i=0}^n f^{-i}(T(f^i(y),0,\epsilon))$ forms the $(n,c(\epsilon))$-spanning set we need.

\end{proof}

Let $T(n,\epsilon)=\cup_{y\in S(n,\epsilon)}T(y,n,\epsilon)$. Then $|T(n,\epsilon)|<\frac{2nK_2}{c(\epsilon)}|S(n,\epsilon)|$ and one can easily show that $T(n,\epsilon)$ is $(n,2c(\epsilon))$-spanning for $A(x,\delta)\cap\Lambda$ (because of the local product structure of $\Lambda$).

So $b(f,n,2c(\epsilon),A(x,\delta)\cap\Lambda)<\frac{2nK_2}{c(\epsilon)}b(f,n,\epsilon,W^u_{\delta}(x)\cap\Lambda)$, and taking the logarithm, dividing by $n$, and taking the limit for $n\rightarrow\infty$, we get that 
$$b(f,2c(\epsilon),A(x,\delta)\cap\Lambda)\leq b(f,\epsilon,W^u_{\delta}(x)\cap\Lambda).$$ 
Now taking $\epsilon\rightarrow 0$ and using that fact that $\lim_{\epsilon\rightarrow 0}c(\epsilon)=0$ we get that $h(f,A(x,\delta)\cap\Lambda)\leq h(f,W^u_{\delta}(x)\cap\Lambda)$.

The proof that $h(f,D(x,\delta)\cap\Lambda)=h(f,A(x,\delta)\cap\Lambda)$ is similar (in fact it is easier, one uses the uniform contraction of the stable foliation).
\end{proof}

\section{Proof of Theorem \ref{t.mainA}}

In this section we will give the proof of Theorem \ref{t.mainA}. We fix $f\in\cU$. First we need to improve Lemma \ref{localholonomy}.

\begin{Lem}\label{l.holonomy}
Given $K>0$, there exist $\epsilon_K$, $\gamma_K>0$ and a neighborhood $\cU_f\subset\cU$ of $f$ such that if $g\in\cU_f$ and $y\in \cup_{z\in W^c_K(x,g)}B(z,\gamma_K)$, then the center-stable holonomy $h^{cs}$ is well-defined between $W^u_{\epsilon_K}(x,g)$ and $W^u(y,g)$, and for any $\epsilon<\epsilon_K$, there exist $c_1(\epsilon,K), c_2(\epsilon,K)>0$ such that
$$
W^u_{c_1(\epsilon, K)}(y,g)\subset h^{cs}(W^u_{\epsilon}(x,g))\subset W^u_{c_2(\epsilon, K)}(y,g).
$$
\end{Lem}

\begin{proof}
By Lemma \ref{localholonomy}, the center holonomy is well defined for $f$ and $W^c_K(x)$; one can extend this holonomy 
to a neighborhood of $f$ and a neighborhood of $W^c_K(x)$ using the continuity of the center-stable holonomy (and the compactness of $M$).
The bound $c_1(\epsilon,K)$ and $c_2(\epsilon,K)$ come directly from the compactness.
\end{proof}

Now we start the proof of the main Theorem. The proof is based on the uniform rate of expansion of unstable disks around the invariant sets $h_f(\Lambda_i)$. Because the non-wandering set of $f$ is inside $\cup_{i=1}^kh_f(\Lambda_i)$, by variation principle, $h(f)=h(f,\cup_{i=1}^kh_f(\Lambda_i))$. We can assume without loss of generality that the entropy of $f$ is $h(f)=h(f|_{h_f(\Lambda_1)})$. Denote $h_f(\Lambda_1)=\Lambda_f$. The following is the main lemma which we will use.

\begin{Lem}\label{uniformexpanding}
Let $f$, $\Lambda_f$ be as above, and let $\epsilon>0$. There exist $\delta>0$ and $N(\epsilon)>0$ such that for any $x\in\Lambda_f$, $f^{N(\epsilon)}(W^u_{\delta}(x,f))$ contains at least $\exp^{N(\epsilon)(h(f)-3\epsilon)}$ disjoint sets of the form $W^u_{2\delta}(y_i,f)$, with $y_i\in\Lambda_f$.
\end{Lem}

Let us provide the idea of the proof first. The entropy of $f$ is realized on some $W^u_{\delta}(x_0,f)$ for some $\delta$ small and $x_0\in\Lambda_f$. For any $x\in\Lambda_f$, there exists an $n_1$ (independent of $x$) such that $f^{n_1}(W^u_{\delta}(x,f))$ contains a disk which is the image of $cs$-holonomy of $W^u_{\delta}(x_0,f)$ (the holonomy has the distance at most $K_0$ in the center direction and very small in the stable direction). Then a $(n,\rho)$-separated set in $W^u_{\delta}(x_0,f)$ can be translated to a $(n+n_1,c_1(\rho,\hat{K}_0))$-separated set in $W^u_{\delta}(x,f)$ for some $\hat{K}_0>0$. By iterating $n_2$ more times, we obtain an $(n+n_1+n_2,2\delta)$-separated set. Taking $\rho$ arbitrarily small and $n$ arbitrarily large, and using the fact that the entropy of $f$ is attained on $W^u_{\delta}(x_0,f)$, we will obtain the result.

\begin{proof}

Let us recall that $K_0$ is given in Proposition \ref{centerunstabledense}. By Lemma \ref{nonexpandingcenter}, there is
$\hat{K}_0$ such that for any $x\in M$ and $n\in \mathbb{Z}$, $f^n(W^c_{K_0}(x,f))\subset W^c_{\hat{K}_0}(f^n(x),f)$. We may choose 
$\hat{K}_0>K_0$. Taking $K=\hat{K}_0$ in Lemma \ref{l.holonomy}, there exist $\epsilon_0=\epsilon_{\hat{K}_0}$ and $\gamma_0=\gamma_{\hat{K}_0}$. We claim that there is $L>0$ such that for any $x\in \Lambda_f$, $W^c_{K_0}(W^u_{L}(x,f),f)$ is $\gamma_0$ dense in $\Lambda_f$.
\begin{proof}
For each point $x\in \Lambda_f$, by Proposition \ref{centerunstabledense}, there is $L_x>0$ such that $W^c_{K_0}(W^u_{L(x)}(x,f),f)$ is $\gamma_0$ dense in $\Lambda_f$. In fact, there is a open neighborhood $B_x$ such that for any $y\in B_x$, $W^c_{K_0}(W^u_{L(x)}(x,f),f)$ is $\gamma_0$ dense in $\Lambda_f$. Then $\{B_x\}_{x\in \Lambda_f}$ is an open cover. Take a finite subcover $\{B_{x_i}\}$ and let $L=\max\{L(x_i)\}$.
The proof is complete.

\end{proof}

Take $\delta=\min\{\epsilon_0,\delta_0\}$ where $\delta_0$ is given by Proposition \ref{unstableentropy}.
Because $\Lambda_f$ is compact, it can be covered by finitely many sets $D(x_i,\delta)$ with $x_i\in \Lambda_f$.
Take a finite partition of $\Lambda_f=\cup_{i=1}^l P_i$ such that each component of this partition is contained in one of the above 
open set $D(x_j,\delta)$. By Remark \ref{piece}, there is $1\leq i \leq l$ such that $h(f)=h(f,P_i)$. Suppose that
$P_i$ is contained in $D(x_0,\delta)$, then $h(f)=h(f,D(x_0,\delta))$ for some $x_0\in \Lambda_f$. 
Then by Proposition \ref{unstableentropy}, $h(f)=h(f,W^u_{\delta}(x_0,f))$.

There is $n_1$ such that for any $x\in \Lambda_f$, $f^{n_1}(W^u_{\delta}(x,f))\supset W^u_{L+c_2(\delta,K_0)}(f^n(x),f)$. By the claim above,
$W^c_{K_0}(f^{n_1}(W^u_{\delta}(x,f)))$ is $\gamma_0$ dense in $\Lambda_f$. By Lemma \ref{l.holonomy}, there is a center stable 
holonomy map $h^{cs}_{x_0}$ between $W^u_{\delta}(x_0,f)$ and $W^u_{c_2(\delta,K_0)}(z,f)\subset f^{n_1}(W^u_{\delta}(x,f))$. 

Choose $0<\rho<\delta$ such that $$a(f,\rho,W^u_{\delta}(x_0,f))>h(f, W^u_{\delta}(x_0,f))-\epsilon=h(f)-\epsilon.$$

Take $n_2$ such that for any $y\in \Lambda_f$, $f^{n_2}(W^u_{c_1(\rho/2,K_0)}(y,f))\supset W^u_{2\delta}(f^{n_2}(y),f)$ where $c_1(\rho/2,K_0)$ is given by Lemma \ref{l.holonomy}. 
There exists $n_{\epsilon}$, large enough so 
$$\frac{n_\epsilon}{n_1+n_{\epsilon}+n_2}>\frac{h(f)-3\epsilon}{h(f)-2\epsilon},$$ 
and such that we have $a(f,n_\epsilon,\rho, W^u_{\delta}(x_0,f))>\exp^{n_{\epsilon}(h(f)-2\epsilon)}$. Let $N(\epsilon,\delta)=n_1+n_\epsilon +n_2$.

Let $S$ be a maximal $(n_{\epsilon},\rho)$-separated set in $W^u_{\delta}(x_0,f)$. Then $$|S|=a(f,n_{\epsilon},\rho_0,W^u_{\delta_0}(x_0,f))>\exp^{n_{\epsilon}(h(f)-2\epsilon)},$$ 
and by the choice of $n_{\epsilon}$ we get $|S|>\exp^{N(\epsilon,\delta)(h(f)-3\epsilon)}$.

Because $W^u$ is expanding, this means that $\forall a,b\in S$, $d_u(f^{n_{\epsilon}}(a),f^{n_{\epsilon}}(b))>\rho$. Then $W^u_{\rho/2}(f^n(a),f)\cap W^u_{\rho/2}(f^n(b),f)=\emptyset$. 
Denote by $a^{'}=h^{cs}_{x_0}(a)\in f^{n_1}(W^u_{\delta}(x,f))$, then there is a piecewise smooth curve $l(a)=l^c(a)\cup l^s(a)$ connecting $a$ and $a^{'}$ such that $l^c(a)\subset W^c_{K_0}(a^{'},f)$ and $l^s(a)\subset W^s_{\gamma_0}(a,f)$. This implies that $f^{n_{\epsilon}}(a)$ and $f^{n_{\epsilon}}(a^{'})$ is connected by $f^{n_\epsilon}(l(a))=f^{n_\epsilon}(l^c(a)) \cup f^{n_\epsilon}(l^s(a))$, where $f^{n_\epsilon}(l^c(a))\subset W^c_{\hat{K}_0}(f^{n_\epsilon}(a))$ by Lemma \ref{nonexpandingcenter}, and $f^{n_\epsilon}(l^s(a))\subset W^s_{\gamma_0}(f^{n_\epsilon}(a^{'}))$ by the uniformly contraction along the strong stable bundle. Hence, by Lemma \ref{l.holonomy},
there is a holonomy map $h^{cs}_{f^{n_\epsilon}(a)}= f^{n_\epsilon}\circ h^{cs}_{x_0}\circ f^{-n_\epsilon}$ from 
$W^{u}_{\rho/2}(f^{n_\epsilon}(a),f)$ to $W^u_{c_2(\rho/2,\hat{K}_0)}(f^{n_\epsilon}(a^{'}),f)$, such that
$$h^{cs}_{f^{n_\epsilon}(a)}(W^u_{\rho/2}(f^{n_\epsilon}(a),f))\supset W^u_{c_1(\rho/2,\hat{K}_0)}(f^{n_\epsilon}(a^{'}),f).$$

Then, there is a family of distinct unstable disks $\{W^u_{c_1(\rho/2,\hat{K}_0)}(f^{n_\epsilon}(a^{'}),f)\}_{a\in S}$, which are contained in $f^{n_1+n_\epsilon}(W^u_{\delta}(x,f))$. By the choice of $n_2$ and definition of $N(\epsilon, \delta)$, 
$f^{N(\epsilon,\delta)}(W^u_{\delta}(x,f))$ contains $\exp^{N(\epsilon)(h(f)-3\epsilon)}$ disjoint sets
$$\{W^u_{2\delta}(f^{n_\epsilon+n_2}(a^{'}),f)\}_{a^{'}\in h^{cs}_{x_0}(S)}.$$
And by the product structure of $\Lambda_f$, $f^{n_\epsilon+n_2}(a^{'})\in \Lambda_f$. The proof is complete.

\end{proof}

Using the continuity of the foliations we get the following corollary.

\begin{Cor}
Let $f$, $\Lambda_f$ be as above, and let $\epsilon, \delta>0$, and $N(\epsilon,\delta)$ given by the previous lemma. There exists a neighborhood $\cU_f(\epsilon)\subset\cU$ of $f$ such that for any $g\in\cU_f(\epsilon)$, and for any $x\in\Lambda_g$ (corresponding to $\Lambda_f$), then $g^{N(\epsilon,\delta)}(W^u_{\delta}(x,g)$ contains at least $\exp^{N(\epsilon,\delta)(h(f)-3\epsilon)}$ disjoint sets of the form $W^u_{\delta}(y_i,g)$, with $y_i\in\Lambda_g$.
\end{Cor}

Now we can finish the proof of the main theorem. By induction we can show that for any $g\in\cU_f(\epsilon)$, any $x\in\Lambda_g$, and any $l>0$, then $g^{lN(\epsilon,\delta)}(W^u_{\delta}(x,g)$ contains at least $\exp^{lN(\epsilon,\delta)(h(f)-3\epsilon)}$ disjoint sets of the form $W^u_{\delta}(y_i,g)$ for $y_i\in \Lambda_g$. But this gives rise to a $(lN(\epsilon,\delta),\delta)$-separated set of $W^u_{\delta}(x,g)$ of cardinality greater or equal to $\exp^{lN(\epsilon,\delta)(h(f)-3\epsilon)}$, so $h(g)\geq a(g,\delta,W^u_{\delta}(x,g))\geq h(f)-3\epsilon$, which shows the lower semicontinuity of the entropy function $h$.

The upper semicontinuity for partially hyperbolic diffeomorphisms with one-dimensional center bundle is proved in \cite{LVY}. We complete the proof of the Theorem \ref{t.mainA}.

\section{Appendix: other partially hyperbolic diffeomorphisms}
Here we show that the topological entropy depends continuously in the case of the other known partially hyperbolic diffeomorphisms with one dimensional center mentioned in Introduction.

\begin{enumerate}
\item[(1)]
On the uniformly hyperbolic diffeomorphisms case, this diffeomorphism is structure stable, i.e., it is conjugate to nearby diffeomorphisms. Hence, the topological entropy is locally constant.
\item[(2)]
If $f$ is a skew products over (topologically) uniformly hyperbolic base, as an immediate corollary of Ledrappier-Walters variational principle \cite{LW}, the topological entropy of $f$ and nearby diffeomorphisms coincide to the topological entropy on the base, which is
topologically uniformly hyperbolic homeomorphisms, and the topological entropy is locally constant. For a more complete discussion see \cite{HHTU}.

Let us note that, by the work of Parwani, Hammerlindl and Potrie (\cite{P, H13b, HP14}), this case contains all partially hyperbolic diffeomorphisms on 3 dimension Nilmanifold where the manifold is different to $T^3$.

\item[(3)]
A 3 dimensional partially hyperoblic diffeomorphism over $T^3$ is called derived Anosov if it is in the same isotopy class with
a linear Anosov diffeomorphism. The first derived Anosov example was given by Ma\~n\'e \cite{M} to construct the example of diffeomorphisms which are transtive but not hyperoblic. In the Ma\~n\'e's example, the diffeomorphisms are $C^0$ close to the linear Anosov diffeomorphism. Recent study on general derived Anosov diffeomorphisms can be found in \cite{BBI,H13a,Potrie,HP14,FPS}.

The topological entropy of a 3 dimensional derived from Anosov diffeomorphism has the same topological entropy as the linear Anosov diffeomorphism. For the original Ma\~n\'e's example \cite{M}, this fact was proved by \cite{BFSV}. For general derived Anosov diffeorphisms, this was observed by Ures in \cite{U} (for a generalization to higher dimension see \cite{FPS}).

\item[(4)]
The non-dynamically coherent 3 dimensional examples of Hertz-Hertz-Ures \cite{HHU} is in fact an Axiom A diffeomorphism with a torus attractor and
a torus repeller. By the $\Omega$ stability theorem, its non-wandering set is conjugate to the non-wandering set of nearby diffeomorphisms. This implies that the topological entropy is locally constant.

\item[(6)]
Let us observe that the new 3 dimensional examples by Bonatti-Parwani-Potrie are perturbations of a partially hyperbolic 
diffeomorphism $f_0$ where $f_0$ has a unique attractor $A$ and a unique repeller $R$, and $f|_{A\cup R}$ is the time-one map of a hyperbolic flow $X_t$ restricted to basic pieces (\cite{BPP}[Theorem 1.1]). It is shown that $f_0$ is dynamically coherent and the center foliaiton is plaque expansive (definition see \cite{HPS}). Hence, by \cite{HPS},
there is a $C^1$ neighborhood $\cU$ of $f_0$ such that every diffeomorphism $f\in \cU$ is dynamically coherent. Moreover, the center foliation of $f$ is structural stable, there is a homeomorphism $h_f$ such that for any $x\in M$, $h_f(W^c(x,f_0))=W^c(h_f(x),f)$ and 
$d_{C^0}(h,Id)$ is close to zero.

Denote by $A_f=h_f(A)$ and $R_f=h_f(R)$, both $A_f$ and $R_f$ are compact invariant and center saturated subset of $f$.
Because $A$ (resp. $R$) is an attractor (resp. repeller) for hyperbolic flow $X_t$, it is $cu$ (resp. $cs$) saturated, which means, contain entire center unstable (resp. center stable) leaves of the corresponding diffeomorphism. Then by the leaf conjugacy,
$A_f$ and $R_f$ are $cu$ and $cs$ saturated respectively.

\begin{Lem}
For any $f\in \cU$, the non-wandering set of $f$ is contained in $A_f\cup R_f$.
\end{Lem}

\begin{proof}

Here we only sketch the proof. Take a small neighborhood $U$ of $A$, by continuiation, for $\cU$ sufficiently small, $A_f\subset U$.
It suffices to show that for every $x\in U\setminus A_f$, there is
$n(x)$ such that $f^{-n(x)}(x)\notin U$.

Note that $A_f$ is $W_f^{cu}$ saturated, take $\epsilon>0$ such that $U\subset \cup_{y\in A_f}W^s_{\epsilon}(y,f)$. For $z\in U$, define function 
$$k(z)=\min_{y\in A_f\cap W^s(z,f)}\{d_s(y,z)\}.$$

Then for any $z\in U\setminus \Lambda_f$, $0<k(z)<\epsilon$. There is $0<\lambda<1$ such that $\|Df|_{E^s(y)}\|<\lambda$ for any $y\in U$, then if $f^{-1}(x)\in U$, we have $k(f^{-1}(x))>\frac{\epsilon}{\lambda}$. And by induction, if suppose $f^{-i}(x)\in U$ for every $i\geq 0$, $k^s(f^{-i}(x))\to \infty$, which is a contradiction.

The proof is complete.

\end{proof}

Because $A_f$ is $cu$ saturated, it is easy to check that the previous arguments for flow case still work, we conclude that $h(f,A_f)$ varies continuously respect to $f\in \cU$. For $R_f$, we will restate the proof by using its center stable leaves.

\end{enumerate}

Let us remark in the end that the method used in this paper for the proof of the continuity of the entropy can be generalized to other partially hyperbolic diffeomorphisms with the dimension of the center equal to one. The fact that we considered perturbations of time-one maps of Anosov flows made the arguments easier, but in fact all that is needed is that the expansion of different unstable disks of points from the non-wandering set are comparable, and this should be true under the assumption of non-expanding center, dynamical coherence, and some kind of `weak' transitivity for example.

\end{document}